\documentclass[11pt,a4paper,thmdefs]{amsart}
\usepackage{amsfonts}
\usepackage{amssymb,amsmath}
\usepackage{amscd}
\usepackage{graphicx}

\newtheorem{theorem}{Theorem}[section]

\newtheorem{corollary}{Corollary}[section]
\newtheorem{definition}{Definition}[section]

\numberwithin{equation}{section} \textwidth12.5cm
\input{tcilatex}

\begin{document}
\title[]{On a new Type Mannheim Curve}
\author[]{\c{C}ET\.{I}N CAMCI}
\address{Department of Mathematics \\
Onsekiz Mart University \\
17020 \c{C}anakkale, Turkey}
\email{ccamci@comu.edu.tr}
\date{}

\begin{abstract}
In this paper, we define a new type curve as $V-$Mannheim curve, $V-$%
Mannheim partner curve and generating curve of Mannheim curve. We give
characterization of \ these curve. In addition, \ we study a relation
between Mannheim curve and spherical curve. Eventually, with Salkowski
method, we give an example of the Mannheim curve.
\end{abstract}

\maketitle


\section{\textbf{Introduction}}

Regle surface plays an important role in the applied science. The condition
that the principal normal of the based curve of a one regle surface may be
the principal normal of the based curve of a second regle surface. This
problem proposed by Saint-Venant and solved by Bertrant. The generalized of
the Saint-Venant and Bertrant problem have been studied by many geometers.
In this space, we can study six cases to be considered(\cite{JM}). The other
important condition that the principal normal of $\left( \alpha \right) $
curve may be the binormal of $\left( \beta \right) $ curve. If the condition
satify between corresponded point, then it is said that $\left( \alpha
\right) $ is Mannheim curve and $\left( \beta \right) $ is Mannheim partner
curve (\cite{Liu}). The curve $\left( \alpha \right) $ is Mannheim curve if
and only if $\kappa =\lambda \left( \kappa ^{2}+\tau ^{2}\right) $where $%
\kappa $, $\tau $ \ are curveture of the curve $\left( \alpha \right) $ and $%
\lambda $ is nonzero constant(\cite{JM}). The curve $\left( \beta \right) $
is Mannheim partner curve if and only if 
\begin{equation*}
\frac{d\overline{\tau }}{d\overline{s}}=\frac{\overline{\kappa }}{\lambda }%
\left( 1+\lambda ^{2}\overline{\tau }^{2}\right)
\end{equation*}%
where $\overline{\kappa }$, $\overline{\tau }$ are curveture of the curve $%
\left( \beta \right) $ and $\lambda $ is nonzero constant(\cite{JM},\cite%
{Liu}). After Liu and Wang paper (\cite{Liu}), many geometer have studied a
Mannheim curve and Mannheim partner curve(\cite{SO}, \cite{KO}, \cite{IS}, 
\cite{OZT}). In this paper, we define $V-$Mannheim curve and $V-$Mannheim
partner curve we give characterization of the $V-$Mannheim curve and the $V-$%
Mannheim partner curve.

\section{\textbf{Preliminaries}}

IIn 3-Euclidean spaces, let $\gamma :I\longrightarrow 
\mathbb{R}
^{3}(s\rightarrow \gamma (s))$ be a regular curve with unit speed coordinate
neighborhood $(I,\gamma )$. Derivation of the Serret-Frenet vectors field
given by 
\begin{equation*}
\left( 
\begin{array}{c}
T%
{\acute{}}
\\ 
N^{%
{\acute{}}%
} \\ 
B^{%
{\acute{}}%
}%
\end{array}%
\right) =\left( 
\begin{array}{ccc}
0 & \kappa & 0 \\ 
-\kappa & 0 & \tau \\ 
0 & -\tau & 0%
\end{array}%
\right) \left( 
\begin{array}{c}
T \\ 
N \\ 
B%
\end{array}%
\right)
\end{equation*}%
where $\left\{ T,N,B\right\} $ is ortonormal Serret-Frenet frame of the
curve and $\kappa $ and \ $\tau $ are curvatures of the curves(\cite{F}, 
\cite{S}). let $\gamma $ be a regular spherical curve. Hence we can define a
curve $\alpha :I\longrightarrow 
\mathbb{R}
^{3}(s\rightarrow \alpha (s))K$ as 
\begin{equation}
\alpha (s)=\int S_{M}(s)\gamma (s)ds
\end{equation}%
where $S_{M}:I\longrightarrow 
\mathbb{R}
$ $\left( s\rightarrow S_{M}(s)\right) $ is differentiable function (\cite%
{Cam}). \ The curve $\alpha $ is spherical curve if and only if 
\begin{equation*}
S_{M}(s)=\left\Vert \gamma ^{\prime }(s)\right\Vert \cos \left( \overset{s}{%
\underset{0}{\dint }}\frac{\det (\gamma (s),\gamma ^{\prime }(s),\gamma
^{\prime \prime }(s))}{\left\Vert \gamma ^{\prime }(s)\right\Vert ^{2}}%
ds+\theta _{0}\right) .
\end{equation*}%
(\cite{Cam}). So, there is \ $S_{T}:I\longrightarrow 
\mathbb{R}
$ differentiable function such that 
\begin{equation*}
\left\Vert \int S_{T}(s)\gamma ^{\prime }(s)ds\right\Vert =1
\end{equation*}%
where 
\begin{equation}
S_{T}(s)=\kappa (s)\cos \left( \overset{s}{\underset{0}{\dint }}\tau
(u)du+\theta _{0}\right)  \label{4.8}
\end{equation}%
If we define a curve $K$ with coordinate neighborhood $(I,\beta )$ such that 
\begin{equation*}
\beta ^{\prime }(s)=\int S_{T}(s)\gamma ^{\prime }(s)ds
\end{equation*}%
then we have 
\begin{equation}
\beta ^{\prime \prime }(s)=S_{T}(s)\gamma ^{\prime }(s)  \label{4.9}
\end{equation}%
Arc-lenght parameter of $M$ and $K$ are same. Let $\left\{ \overline{T},%
\overline{N},\overline{B},\overline{\kappa },\overline{\tau }\right\} $ be
Serret-Frenet apparatus of the curves where 
\begin{equation}
S_{T}(s)=\overline{\kappa }(s)=\kappa (s)\cos \left( \overset{s}{\underset{0}%
{\dint }}\tau (u)du+\theta _{0}\right)  \label{4.91}
\end{equation}%
and 
\begin{equation}
\overline{\tau }(s)=\kappa (s)\sin \left( \overset{s}{\underset{0}{\dint }}%
\tau (u)du+\theta _{0}\right) \text{.}  \label{4.10}
\end{equation}%
From (\ref{4.9}), we have 
\begin{equation}
\overline{N}(s)=\varepsilon T(s)  \label{4.11}
\end{equation}%
where $\varepsilon =\pm 1$. From equation (\ref{4.9}) and (\ref{4.11}), we
can see that principal normal of $K$ and tangent of $M$ is colinear.

\section{\textbf{\ V-Mannheim Curve}}

In 3-Euclidean spaces, let $\gamma :I\longrightarrow 
\mathbb{R}
^{3}(s\rightarrow \gamma (s))$ be a regular curve with unit speed coordinate
neighborhood $(I,\gamma )$ and $\left\{ T,N,B\right\} $ be ortonormal frame
of the curve and $\kappa $ and \ $\tau $ be curvatures of the curves. Choi
and Kim defined a unit vector field $V$ given by 
\begin{equation*}
V(s)=u(s)T(s)+v(s)N(s)+w(s)B(s)
\end{equation*}%
Integral curve of $V$ is define by $\gamma _{V}(s)=\int V(s)ds$ where $u,v,w$
are founctions from $I$ to $%
\mathbb{R}
$(\cite{Choi}). So we can defined a curve $\beta $ as 
\begin{equation}
\beta (s)=\int V(s)ds+\lambda (s)N(s)  \label{a}
\end{equation}%
where $\lambda :I\longrightarrow 
\mathbb{R}
$ $\left( s\rightarrow \lambda (s)\right) $ is diferentiable founction. Let $%
\left\{ \overline{T},\overline{N},\overline{B}\right\} $ be ortonormal frame
of the curve and $\overline{\kappa }$ and \ $\overline{\tau }$ be curvatures
of the curves $\beta $. Under the above notation, we can give the following
theorems and definitions.

\begin{definition}
If $\left\{ N,\overline{B}\right\} $ is lineer dependent $\left( \overline{B}%
=\epsilon N,\epsilon =\pm 1\right) $, then it is said that the curve $\gamma 
$(respectively $\beta )$ is said that $V-$Mannheim curve ( $V-$Mannheim
partner curve). By similar method, Camc\i\ defined a $V-$Bertrant curve (%
\cite{CAM}).

\begin{theorem}
The curve $\gamma $ is a $V-$Mannheim curve if and only if it is satisfy
that 
\begin{equation}
u\kappa -w\tau =\lambda \left( \kappa ^{2}+\tau ^{2}\right)  \label{3.7}
\end{equation}%
and 
\begin{equation}
\lambda (s)=-\int v(s)ds  \label{3.8}
\end{equation}%
where $\lambda _{0}$ is constant.

\begin{proof}
If $M$ is a $V-$Mannheim curve, then $V-$Mannheim partner curve of $M$ is
equal to 
\begin{equation}
\beta (s)=\overset{s}{\underset{0}{\dint }}V(u)du+\lambda (s)N(s)
\label{3.9}
\end{equation}%
and$\left\{ N,\overline{B}\right\} $ is a lineer dependent. \ If we derivate
the equation (\ref{3.9}), we have 
\begin{equation*}
\frac{d\overline{s}}{ds}\overline{T}=\left( u-\lambda \kappa \right)
T+\left( \lambda ^{\prime }+v\right) N+\left( w+\lambda \tau \right) B.
\end{equation*}%
Since $\left\{ N,\overline{B}\right\} $ is a lineer dependent, we have 
\begin{equation*}
\lambda (s)=-\int v(s)ds
\end{equation*}%
hence we get 
\begin{equation}
\overline{T}=\frac{ds}{d\overline{s}}\left( u-\lambda \kappa \right) T+\frac{%
ds}{d\overline{s}}\left( w+\lambda \tau \right) B=\cos \theta (s)T+\sin
\theta (s)B  \label{3.10}
\end{equation}%
where $\cos \theta (s)=\frac{ds}{d\overline{s}}\left( u-\lambda \kappa
\right) $ and $\sin \theta (s)=\frac{ds}{d\overline{s}}\left( v+\lambda \tau
\right) $. So we have 
\begin{equation}
\tan \theta (s)=\frac{v+\lambda \tau }{u-\lambda \kappa }  \label{3.11}
\end{equation}%
If we derivate the equation (\ref{3.10}), we have 
\begin{equation}
\frac{d\overline{s}}{ds}\overline{\kappa }\overline{N}=-\theta ^{\prime
}\sin \theta T+\left( \kappa \cos \theta -\tau \sin \theta \right) N+\theta
^{\prime }\cos \theta B  \label{3.12}
\end{equation}%
Since $\left\{ N,\overline{B}\right\} $ is a lineer dependent, \ we get 
\begin{equation}
\kappa \cos \theta -\tau \sin \theta =0  \label{3.13}
\end{equation}%
From equation (\ref{3.11}) and (\ref{3.13}), we have 
\begin{equation}
u\kappa -w\tau =\lambda \left( \kappa ^{2}+\tau ^{2}\right) .  \label{3.14}
\end{equation}%
Conversely, we define a curve as 
\begin{equation}
\beta (s)=\int V(s)ds+\lambda (s)N(s)  \label{3.15}
\end{equation}%
where $\lambda :I\longrightarrow 
\mathbb{R}
$ $\left( s\rightarrow \lambda (s)\right) $ is diferentiable founction such
that $\lambda (s)=-\int v(s)ds$ . If we derivate the equation (\ref{3.15}),
we have%
\begin{equation}
\frac{d\overline{s}}{ds}\overline{T}=\left( u-\lambda \kappa \right)
T+\left( w+\lambda \tau \right) B.  \label{3.16}
\end{equation}%
From equation (\ref{3.16}), we get 
\begin{equation}
\overline{T}=\frac{ds}{d\overline{s}}\left( u-\lambda \kappa \right) T+\frac{%
ds}{d\overline{s}}\left( w+\lambda \tau \right) B=\cos \theta (s)T+\sin
\theta (s)B  \label{3.17}
\end{equation}%
where $\cos \theta (s)=\frac{ds}{d\overline{s}}\left( u-\lambda \kappa
\right) $ and $\sin \theta (s)=\frac{ds}{d\overline{s}}\left( v+\lambda \tau
\right) $. From equation (\ref{3.9}) and (\ref{3.12}) we have 
\begin{equation}
\tan \theta (s)=\frac{v+\lambda \tau }{u-\lambda \kappa }  \label{3.18}
\end{equation}%
Using equation (\ref{3.17}), we obtain 
\begin{equation}
\frac{d\overline{s}}{ds}\overline{\kappa }\overline{N}=-\theta ^{\prime
}\sin \theta T+\left( \kappa \cos \theta -\tau \sin \theta \right) N+\theta
^{\prime }\cos \theta B  \label{3.19}
\end{equation}%
From equation (\ref{3.18}) and (\ref{3.19}), we get%
\begin{equation}
\kappa \cos \theta -\tau \sin \theta =\frac{\cos \theta }{u-\lambda \kappa }%
\left( u\kappa -v\tau -\lambda \left( \kappa ^{2}+\tau ^{2}\right) \right) =0
\label{3.20}
\end{equation}%
So, we have 
\begin{equation}
\overline{N}=-\sin \theta T+\cos \theta B  \label{3.21}
\end{equation}%
and 
\begin{equation}
d\overline{s}\text{ }\overline{\kappa }=\theta ^{\prime }ds=d\theta
\label{3.22}
\end{equation}%
From equation (\ref{3.17}) and (\ref{3.21}), we see that $\left\{ N,%
\overline{B}\right\} $ is a lineer dependent.
\end{proof}

\begin{corollary}
If $u(s)=1,v(s)=w(s)=0$, we have Mannheim ($T$-Mannheim)\ curve. From
equation (\ref{3.8}), $\lambda $ is constant. \ From equation (\ref{3.7}),
we have $\ \kappa =\lambda \left( \kappa ^{2}+\tau ^{2}\right) $. If $%
v(s)=1,u(s)=w(s)=0$, we have $B$-Mannheim curve. Using equation (\ref{3.7}),
we have $\tau =-\lambda \left( \kappa ^{2}+\tau ^{2}\right) $ where $\lambda 
$ is constant.

\begin{theorem}
In 3-Euclidean spaces, let $M$ be a regular curve with coordinate
neighborhood $(I,\gamma )$ and "$s$" be arcparameter of the curve. Let $%
\left\{ T,N,B\right\} $ be ortonormal frame of the curve and $\kappa $ and \ 
$\tau $ be curvatures of the curves . $M$ is a $V-$Mannheim curve if and
only if it is satisfy that 
\begin{equation}
2u\kappa (s)=\frac{1}{\lambda }\left[ u^{2}+u\sqrt{u^{2}+w^{2}}\cos \left(
2\theta +\theta _{0}\right) \right]  \label{3.23}
\end{equation}%
and 
\begin{equation}
2w\tau (s)=\frac{1}{\lambda }\left[ -w^{2}+w\sqrt{u^{2}+w^{2}}\sin \left(
2\theta +\theta _{0}\right) \right]  \label{3.24}
\end{equation}%
where $\cos \theta _{0}=\frac{u}{\sqrt{u^{2}+w^{2}}}$and $\sin \theta _{0}=%
\frac{w}{\sqrt{u^{2}+w^{2}}}$.
\end{theorem}
\end{corollary}

\begin{proof}
If $M$ is a $V-$Mannheim curve, From equation (\ref{3.7}), then we have $%
u\kappa -w\tau =\lambda \left( \kappa ^{2}+\tau ^{2}\right) $. So we have 
\begin{equation}
\kappa (s)=\sqrt{\frac{u\kappa -w\tau }{\lambda }}\cos \theta  \label{3.25}
\end{equation}%
and%
\begin{equation}
\tau (s)=\sqrt{\frac{u\kappa -w\tau }{\lambda }}\sin \theta  \label{3.26}
\end{equation}%
From equation (\ref{3.25}) and (\ref{3.26}), we have%
\begin{equation}
u\kappa -w\tau =\frac{1}{\lambda }\left( u\cos \theta -w\sin \theta \right)
^{2}  \label{3.27}
\end{equation}%
and 
\begin{equation}
u\kappa +w\tau =\frac{1}{\lambda }\left( u^{2}(\cos \theta )^{2}-w^{2}(\sin
\theta )^{2}\right)  \label{3.28}
\end{equation}%
From equation (\ref{3.27}) and (\ref{3.28}), we have equation (\ref{3.23})
and (\ref{3.24})
\end{proof}

\begin{corollary}
If $u(s)=1,v(s)=w(s)=0$, we have Mannheim$\left( T\text{-Mannheim}\right) $
curve. From equation (\ref{3.8}), $\lambda $ is constant. \ In this case, we
have 
\begin{equation}
\kappa (s)=R(\cos \theta )^{2}  \label{3.29}
\end{equation}%
and 
\begin{equation}
\tau (s)=R\cos \theta \sin \theta  \label{3.30}
\end{equation}%
where $R=\frac{1}{\lambda }$ is constant. From equation (\ref{3.7}), we have 
$\ \kappa =\lambda \left( \kappa ^{2}+\tau ^{2}\right) $. If $%
w(s)=1,u(s)=v(s)=0$, we have $B$-Bertrant curve. From equation (\ref{3.8}), $%
\lambda $ is constant. In this case we have 
\begin{equation}
\kappa (s)=R\cos \theta \sin \theta  \label{3.31}
\end{equation}%
and 
\begin{equation}
\tau (s)=R\left( \cos \theta \right) ^{2}  \label{3.32}
\end{equation}%
where $R=\frac{1}{\lambda }$ is constant.
\end{corollary}
\end{theorem}
\end{definition}

In 3-Euclidean spaces, let $M$, $K$ be a regular curve with unit coordinate
neighborhood $(I,\alpha )$ and $(I,\beta )$ and "$s$" and "$\overline{s}$"
be arcparameter of $M$ and $K,$ respectively. Let $\left( T,N,B,\kappa ,\tau
\right) $ and $\left( \overline{T},\overline{N},\overline{B},\overline{%
\kappa },\overline{\tau }\right) $ be Frenet apparatus of $M$ and $K,$
respectively. Let $V=uT+vN+wB$ be unit vector field where $u,v,w$ are
constant.

\begin{theorem}
Under the above notation, the curve $\left( \beta \right) $ is a $V-$%
Mannheim partner curve if and only if 
\begin{equation}
\frac{d\overline{\tau }}{d\overline{s}}=\frac{v\overline{\tau }\sqrt{%
1+\lambda ^{2}\overline{\tau }^{2}}}{\lambda \sqrt{1-v^{2}}}-\frac{\overline{%
\kappa }}{\lambda }\left( 1+\lambda ^{2}\overline{\tau }^{2}\right)
\label{b}
\end{equation}%
where $\overline{\kappa }$, $\overline{\tau }$ are curveture of the curve $%
\left( \beta \right) $ and $\lambda (s)=-\int v(s)ds$.

\begin{proof}
Let $\left( \beta \right) $ be $V-$Mannheim partner curve of $\left( \alpha
\right) $. So we can write $\overline{B}=N$. From equation (\ref{a}), we can
write 
\begin{equation}
\overset{s}{\underset{0}{\dint }}V(u)du=\beta (s)-\lambda \overline{B}(s)
\label{3.33}
\end{equation}%
If we derivate the equation (\ref{3.33}), we have 
\begin{equation*}
\frac{ds}{d\overline{s}}\left( uT+vN+wB\right) =\overline{T}+\lambda 
\overline{\tau }\overline{N}-\frac{d\lambda }{d\overline{s}}\overline{B}(s)
\end{equation*}%
where $\overline{B}=N$ and $\lambda (s)=-\int v(s)ds$. So we obtain%
\begin{equation}
\frac{ds}{d\overline{s}}\left( uT+wB\right) =\overline{T}+\lambda \overline{%
\tau }\overline{N}  \label{3.34}
\end{equation}%
where 
\begin{equation}
\overline{T}=\frac{ds}{d\overline{s}}\left( u-\lambda \kappa \right) T+\frac{%
ds}{d\overline{s}}\left( w+\lambda \tau \right) B=\cos \theta T+\sin \theta B
\label{3.35}
\end{equation}%
and 
\begin{equation}
\overline{N}=\sin \theta T-\cos \theta B  \label{3.36}
\end{equation}%
From equation (\ref{3.34}), (\ref{3.35}) and (\ref{3.36}), we have 
\begin{equation}
\frac{ds}{d\overline{s}}u=\cos \theta +\lambda \overline{\tau }\sin \theta
\label{3.37}
\end{equation}%
and 
\begin{equation}
\frac{ds}{d\overline{s}}w=\sin \theta -\lambda \overline{\tau }\cos \theta
\label{3.38}
\end{equation}%
Using equations (\ref{3.37}) and (\ref{3.38}), we obtain 
\begin{equation}
\frac{w}{u}=\frac{\sin \theta -\lambda \overline{\tau }\cos \theta }{\cos
\theta +\lambda \overline{\tau }\sin \theta }  \label{3.39}
\end{equation}%
From equation (\ref{3.39}), we can easily see that 
\begin{equation}
\overline{\tau }=\frac{1}{\lambda }\frac{u\sin \theta -w\cos \theta }{w\sin
\theta +u\cos \theta }  \label{3.40}
\end{equation}%
where 
\begin{equation*}
\left( u\sin \theta -w\cos \theta \right) ^{2}+\left( w\sin \theta +u\cos
\theta \right) ^{2}=u^{2}+w^{2}=1-v^{2}
\end{equation*}%
So, there is $\phi :I\longrightarrow 
\mathbb{R}
$ $\left( s\rightarrow \phi (s)\right) $ diferentiable founction such that 
\begin{equation}
\sqrt{1-v^{2}}\cos \phi =w\sin \theta +u\cos \theta  \label{3.41}
\end{equation}%
and 
\begin{equation}
\sqrt{1-v^{2}}\sin \phi =u\sin \theta -w\cos \theta  \label{3.42}
\end{equation}%
From equation (\ref{3.41}) and (\ref{3.42}), we have 
\begin{equation}
\overline{\tau }=\frac{1}{\lambda }\tan \phi  \label{3.43}
\end{equation}%
and 
\begin{equation}
\frac{d\theta }{d\overline{s}}=\frac{d\phi }{d\overline{s}}=-\overline{%
\kappa }  \label{3.44}
\end{equation}%
From equation (\ref{3.43}) and (\ref{3.44}), we obtain 
\begin{equation}
\frac{d\overline{\tau }}{d\overline{s}}=\frac{v\overline{\tau }\sqrt{%
1+\lambda ^{2}\overline{\tau }^{2}}}{\lambda \sqrt{1-v^{2}}}-\frac{\overline{%
\kappa }}{\lambda }\left( 1+\lambda ^{2}\overline{\tau }^{2}\right) \text{.}
\label{3.45}
\end{equation}%
Conversely, we can define a curve as 
\begin{equation}
\overset{s}{\underset{0}{\dint }}V(u)du=\beta (s)-\lambda \overline{B}(s)
\label{3.46}
\end{equation}%
If we derivate the equation (\ref{3.46}), we have 
\begin{equation}
\frac{ds}{d\overline{s}}\left( uT+vN+wB\right) =\overline{T}+\lambda 
\overline{\tau }\overline{N}-\frac{d\lambda }{d\overline{s}}\overline{B}
\label{3.47}
\end{equation}%
From equation (\ref{3.47}), we obtain%
\begin{equation}
\frac{ds}{d\overline{s}}=\sqrt{\frac{1+\lambda ^{2}\overline{\tau }^{2}}{%
1-v^{2}}}  \label{3.48}
\end{equation}%
and 
\begin{equation}
\frac{ds}{d\overline{s}}\left( uT+wB\right) =\overline{T}+\lambda \overline{%
\tau }\overline{N}  \label{3.49}
\end{equation}%
If we derivate the equation (\ref{3.49}), we have $\frac{d(\lambda \overline{%
\tau })}{d\overline{s}}\overline{\tau }+\frac{d(\overline{\tau })}{d%
\overline{s}}\lambda $ 
\begin{equation}
\frac{d\overline{s}}{ds}\frac{d^{2}s}{d\overline{s}^{2}}\left( \overline{T}%
+\lambda \overline{\tau }\overline{N}\right) +\left( \frac{ds}{d\overline{s}}%
\right) ^{2}\left( u\kappa -w\tau \right) N=-\lambda \overline{\tau }%
\overline{\kappa }\overline{T}+\left( \overline{\kappa }+\left( \frac{%
d(\lambda \overline{\tau })}{d\overline{s}}\right) \right) \overline{N}%
+\lambda \overline{\tau }^{2}\overline{B}  \label{3.50}
\end{equation}%
Using (\ref{3.47}), we obtain%
\begin{equation}
\frac{d\overline{s}}{ds}\frac{d^{2}s}{d\overline{s}^{2}}=-\lambda \overline{%
\tau }\overline{\kappa }  \label{3.51}
\end{equation}%
and 
\begin{equation}
\overline{\kappa }+\left( \frac{d(\lambda \overline{\tau })}{d\overline{s}}%
\right) =-\lambda ^{2}\overline{\tau }^{2}\overline{\kappa }  \label{3.52}
\end{equation}%
From equation (\ref{3.50}), (\ref{3.51}) and (\ref{3.52}), we easily can see
that $\left\{ N,\overline{B}\right\} $ is a lineer dependent.
\end{proof}
\end{theorem}

\begin{corollary}
If $v=0$, then we can see that $\lambda $ is non zero constant and $%
u^{2}+w^{2}=1-v^{2}=1$. From equation (\ref{b}), $\left( \beta \right) $ is
a Mannheim partner curve.
\end{corollary}

If $u,v,w$ are founctions from $I$ to $%
\mathbb{R}
$, then we can give following theorem.

\begin{theorem}
The curve $\left( \beta \right) $ is a $V-$Mannheim partner curve if and
only if 
\begin{equation}
\frac{d\overline{\tau }}{d\overline{s}}=\frac{v\overline{\tau }\sqrt{%
1+\lambda ^{2}\overline{\tau }^{2}}}{\lambda \sqrt{1-v^{2}}}+\frac{1}{%
\lambda }\left( \frac{d\left( \arctan \left( \frac{w}{u}\right) \right) }{d%
\overline{s}}-\overline{\kappa }\right) \left( 1+\lambda ^{2}\overline{\tau }%
^{2}\right)  \label{3.53}
\end{equation}%
where $\overline{\kappa }$, $\overline{\tau }$ are curveture of the curve $%
\left( \beta \right) $ and $\lambda (s)=-\int v(s)ds$.
\end{theorem}

\section{\protect\bigskip \textbf{\ Generating Curve Of The Mannheim curve.}}

In 3-Euclidean spaces, let $M$, $K$ be a regular curve with unit coordinate
neighborhood $(I,\gamma )$ and $(I,\beta )$ and "$s$" be arcparameter of $M$
and $K$. Let $\left( T,N,B,\kappa ,\tau \right) $ and $\left( \overline{T},%
\overline{N},\overline{B},\overline{\kappa },\overline{\tau }\right) $ be
Frenet apparatus of $M$ and $K,$ respectively. \ For all $s\in I$, 
\begin{equation*}
\beta ^{\prime \prime }(s)=\overline{\kappa }\gamma ^{\prime }(s)
\end{equation*}%
and 
\begin{eqnarray}
\overline{\kappa }(s) &=&\kappa (s)\cos \left( \overset{s}{\underset{0}{%
\dint }}\tau (u)du+\varphi _{0}\right)  \label{6.24} \\
\overline{\tau }(s) &=&\kappa (s)\sin \left( \overset{s}{\underset{0}{\dint }%
}\tau (u)du+\varphi _{0}\right)  \notag
\end{eqnarray}%
\qquad

\begin{theorem}
In this case, $K$ is a $V-$Mannheim curve if and only if 
\begin{equation}
\kappa (s)=F(s)\cos \left( \varphi (s)+\phi (s)\right)  \label{6.25}
\end{equation}%
where $F(s)=\lambda \sqrt{u^{2}+w^{2}}$, $\cos \phi (s)=\frac{u}{\sqrt{%
u^{2}+w^{2}}}$, $\sin \phi (s)=\frac{w}{\sqrt{u^{2}+w^{2}}}$ and $\varphi
(s)=\overset{s}{\underset{0}{\dint }}\tau (u)du+\varphi _{0}$ .
\end{theorem}

\begin{proof}
i)Let $K$ be a Mannheim curve. In this case, there exist $\lambda \in $ $%
\mathbb{R}
$ such that $\left( \overline{\kappa }(s)\right) ^{2}+\left( \overline{\tau }%
(s)\right) ^{2}=\lambda \overline{\kappa }(s)$ , where "$s$" is arcparameter
of $K$. From equation (\ref{6.24}), we have 
\begin{equation}
u\kappa (s)\cos \varphi (s)-w\kappa (s)\sin \varphi (s)=\lambda \left(
\kappa (s)\right) ^{2}  \label{6.26}
\end{equation}%
From equation (\ref{6.26}), we obtain%
\begin{equation*}
\kappa (s)=F(s)\cos \left( \varphi (s)+\phi (s)\right)
\end{equation*}%
where $F(s)=\frac{\sqrt{u^{2}+w^{2}}}{\lambda }$, $\cos \phi (s)=\frac{u}{%
\sqrt{u^{2}+w^{2}}}$, $\sin \phi (s)=\frac{w}{\sqrt{u^{2}+w^{2}}}$ and $%
\varphi (s)=\overset{s}{\underset{0}{\dint }}\tau (u)du+\varphi _{0}$.
Conversely, From equation (\ref{6.25}), we have%
\begin{equation*}
\lambda \left( u\kappa (s)\cos \varphi (s)-w\kappa (s)\sin \varphi
(s)\right) =\left( \kappa (s)\right) ^{2}
\end{equation*}%
From equation (\ref{6.24}), we get $\lambda \left( \overline{\kappa }%
(s)\right) ^{2}+\left( \overline{\tau }(s)\right) ^{2}=\overline{\kappa }(s)$%
.
\end{proof}

\begin{theorem}
i) $K$ is a Mannheim curve if and only if 
\begin{equation*}
\kappa (s)=R\cos \left( \overset{s}{\underset{0}{\dint }}\tau (u)du+\varphi
_{0}\right)
\end{equation*}%
ii) $K$ is a $B-$Mannheim curve if and only if 
\begin{equation*}
\kappa (s)=R\sin \left( \overset{s}{\underset{0}{\dint }}\tau (u)du+\varphi
_{0}\right)
\end{equation*}%
where $R$ is constant.
\end{theorem}

\begin{proof}
i)Let $K$ be a Mannheim curve. In this case, there exist $\lambda \in $ $%
\mathbb{R}
$ such that $\left( \overline{\kappa }(s)\right) ^{2}+\left( \overline{\tau }%
(s)\right) ^{2}=\lambda \overline{\kappa }(s)$ , where "$s$" is arcparameter
of $K$. From equation (\ref{6.24}), we have $\lambda \left( \left( \overline{%
\kappa }(s)\right) ^{2}+\left( \overline{\tau }(s)\right) ^{2}\right)
=\lambda \left( \kappa (s)\right) ^{2}=\overline{\kappa }(s)$. So we have 
\begin{equation*}
\kappa (s)=R\cos \left( \overset{s}{\underset{0}{\dint }}\tau (u)du+\varphi
_{0}\right)
\end{equation*}%
where $R=\frac{1}{\lambda }=const$. Conversely, Let $M$ be a curve satisfy
that $\kappa (s)=R\cos \varphi (s)$ where $R=\frac{1}{\lambda }=const$ and $%
\varphi (s)=$ $\overset{s}{\underset{0}{\dint }}\tau (u)du+\theta _{0}$.
From equation (\ref{6.24}), we have 
\begin{equation}
\overline{\kappa }(s)=R\left( \cos \varphi \right) ^{2}  \label{6.27}
\end{equation}%
and%
\begin{equation}
\overline{\tau }(s)=R\cos \varphi \sin \varphi  \label{6.28}
\end{equation}%
From equation (\ref{6.25}), we get 
\begin{equation*}
\left( \overline{\kappa }(s)\right) ^{2}+\left( \overline{\tau }(s)\right)
^{2}=R^{2}\left( \cos \varphi (s)\right) ^{2}=R\text{ }\overline{\kappa }(s)
\end{equation*}%
So $K$ is a Mannheim curve.ii) Let $K$ be a $B-$Mannheim curve. In this
case, there exist $\lambda \in $ $%
\mathbb{R}
$ such that $\left( \overline{\kappa }(s)\right) ^{2}+\left( \overline{\tau }%
(s)\right) ^{2}=\lambda \overline{\tau }(s)$ , where "$s$" is arcparameter
of $K$. From equation (\ref{6.24}), we have $\left( \overline{\kappa }%
(s)\right) ^{2}+\left( \overline{\tau }(s)\right) ^{2}=\left( \kappa
(s)\right) ^{2}=\lambda \overline{\tau }(s)$. So we have $\cdot $%
\begin{equation*}
\kappa (s)=R\sin \left( \overset{s}{\underset{0}{\dint }}\tau (u)du+\varphi
_{0}\right)
\end{equation*}%
where $R=\epsilon \lambda =const$. Conversely, Let $M$ be a curve satisfy
that $\kappa (s)=R\sin \varphi (s)$ where $R=\frac{1}{\lambda }=const$ and $%
\varphi (s)=\overset{s}{\underset{0}{\dint }}\tau (u)du+\varphi _{0}$. From
equation (\ref{6.24}), we have 
\begin{eqnarray*}
\overline{\kappa }(s) &=&R\cos \varphi \sin \varphi \\
\overline{\tau }(s) &=&R\left( \cos \varphi \right) ^{2}
\end{eqnarray*}%
From equation (\ref{6.26}), we get 
\begin{equation*}
\left( \overline{\kappa }(s)\right) ^{2}+\left( \overline{\tau }(s)\right)
^{2}=R^{2}\left( \cos \varphi (s)\right) ^{2}=R\text{ }\overline{\tau }(s)
\end{equation*}%
So $K$ is a $B$-Mannheim curve.
\end{proof}

\begin{definition}
\bigskip In 3-Euclidean spaces, let $K$ be a curve satisfy that%
\begin{equation*}
\overline{\kappa }(\overline{s})=R\cos \left( \overset{\overline{s}}{%
\underset{0}{\dint }}\overline{\tau }(u)du+\theta _{0}\right) .
\end{equation*}%
We said that $K$ is a generating curve of the Mannheim curve.
\end{definition}

\end{document}